\documentclass[a4paper,11pt]{article}

\usepackage[english]{babel}
\usepackage{amsmath}
\usepackage{amsfonts}
\usepackage{amssymb}
\usepackage{amsthm}
\usepackage{mathrsfs} 
\usepackage{graphicx}

\usepackage{caption} 

\usepackage{cite} 

\usepackage{microtype}
\usepackage{palatino}
\usepackage[super]{nth}

\usepackage[usenames,dvipsnames]{xcolor} 
\usepackage{enumerate}

 \usepackage{hyperref}
 \hypersetup{
     bookmarks=true,         % show bookmarks bar?
     bookmarksopen=false,    %
     bookmarksopenlevel=1, %
     pdftoolbar=true,        % show Acrobat?s toolbar?
     pdfmenubar=true,        % show Acrobat?s menu?
     pdfstartview={FitH},    % ``FitH'' for fitting the width of the page to the window
     pdftitle={S-cob+Latour},    % title
     pdfauthor={Carlos Moraga Ferr\'{a}ndiz},     % author
     pdfkeywords={Whitehead torsion, $ s $-cobordism, Latour's theorem, Morse-Novikov theory}, % list of keywords
     pdfnewwindow=true,      % links in new window
    colorlinks=false,       % false: boxed links; true: colored links
     linkcolor=Emerald,          % color of internal links
     citecolor=OliveGreen,        % color of links to bibliography
     filecolor=magenta,      % color of file links
     urlcolor=Bittersweet           % color of external links
 }

\newtheorem{Teo}{Theorem}
\newtheorem{df}[Teo]{Definition}
\newtheorem{pro}[Teo]{Proposition}

\newtheorem{rem}[Teo]{Remark}

\providecommand{\keywords}[1]{\textbf{\textit{Keywords: }} #1}
\providecommand{\ack}[1]{\textbf{\textit{Acknowledgements: }} #1}
\providecommand{\MSC}[1]{\textbf{\textit{MSC Classification: }} #1}

\author{C. Moraga Ferr\'{a}ndiz}
\title{The $ s $-cobordism theorem seen as a particular case of Latour's theorem}

\date{}
\newcommand{\longhookrightarrow}{\ensuremath{\lhook\joinrel\relbar\joinrel\rightarrow}}

\newcommand{\longhookleftarrow}{\ensuremath{\leftarrow\joinrel\relbar\joinrel\rhook}}

\DeclareMathOperator{\Id}{Id}

\DeclareMathOperator{\Wh}{Wh}

\DeclareMathOperator{\ev}{ev}

\DeclareMathOperator{\im}{Im} %Parece que colapsa con trfsigns

\DeclareMathOperator{\irr}{irr}

\DeclareMathOperator{\odd}{odd}

\newcommand{\bN}{\mathbb{N}}

\newcommand{\bR}{\mathbb{R}}
\newcommand{\bS}{\mathbb{S}}

\newcommand{\bZ}{\mathbb{Z}}

\newcommand{\cC}{\mathcal{C}}

\newcommand{\cE}{\mathcal{E}}
\newcommand{\cF}{\mathcal{F}}

\newcommand{\cN}{\mathcal{N}}

\newcommand{\cP}{\mathcal{P}}

\newcommand{\cS}{\mathcal{S}}

\newcommand{\al}{\alpha}

\newcommand{\de}{\delta}

\newcommand{\lam}{\lambda}

\newcommand{\Lam}{\Lambda}

\newcommand{\wh}{\widehat}

\newcommand{\wt}{\widetilde}

\newcommand{\nolpp}{\longhookleftarrow}

\newcommand{\plon}{\hookrightarrow}
\newcommand{\plonn}{\longhookrightarrow}

\newcommand{\Sii}{\Leftrightarrow}

\newcommand{\dif}{\cong}

\newcommand{\grad}{\nabla}

\newcommand{\iso}{\approx}

\newcommand{\p}{\partial}

\newcommand{\prive}{\smallsetminus}
\newcommand{\tens}{\otimes}

\newcommand{\x}{\times}
\newcommand{\V}{\varnothing}

\newcommand{\di}{\displaystyle}

\newcommand{\esp}{\,\,\,}

\newcommand{\neee}{\negmedspace}

\providecommand{\flecur}[1]{\ar@/^#1pc/ }

\providecommand{\norm}[1]{\lVert#1\rVert}

\providecommand{\parent}[1]{\left( #1 \right) }

\providecommand{\bparent}[1]{\bigl( #1 \bigr) }

\providecommand{\emb}[1]{\di \mathop{\plon}^{#1}}

\providecommand{\embb}[1]{\di \mathop{\plonn}^{#1}}

\providecommand{\bmee}[1]{\di \mathop{\nolpp}^{#1}}

\providecommand{\sing}[1]{\left\lbrace #1\right\rbrace }

\newcommand{\quotient}[2]{\raise1ex\hbox{$#1$}\Big/ \lower1ex\hbox{$#2$}}

\newcommand{\lift}[4]{\[ 
   \xymatrix{
  & \wt{#2} \ar[d]^{#4}  \\
#1 \ar[r]^#3 \ar@{.>}[ur]^{\wt{#3}} & #2
} 
\]}

\newcommand{\univcolim}[9]{
\[ \xymatrix{
 #1 \ar[rr]^{#4}_{#7}\ar[rd]^{#5}_{#8} & & #2 \\
  & #3 \ar@{.>}[ur]^{#6}_{#9} &
}
\] }

\begin{document}
\maketitle
\begin{abstract}
We show how Latour's theorem (\neee\cite{latour}) can be understood as a natural generalization of the $ s $-cobordism theorem for cohomology classes $ u\in H^1(M;\bR) $. The $ s $-cobordism theorem becomes a special degenerate case when $ u=0 $. 
\end{abstract} 
\\
\MSC{57R80; 19J10}

\section{The $ s $-cobordism theorem: the exact case}
Two connected, closed and oriented manifolds $ N_0^n,N_1^n $ are cobordant if there exists a compact oriented manifold $ W^{n+1} $ such that $ \p W^{n+1}=(-N_0^n)\sqcup (N_1^n) $. Superscripts denote dimension while $ (-N) $ represents the manifold $ N $ with reversed orientation. Such a triad $ (W;N_0,N_1) $ is said to be an $ h $-cobordism if both inclusions $ N_0\embb{i_0} W\bmee{i_1} N_1 $ are homotopy equivalences.\\
Let $ \pi $ be the fundamental group of $ W $; we denote by $ \Lambda:=\bZ[\pi] $ its group ring. To each $ h $-cobordism we can associate its torsion $ \tau (W,N_0) $ which lives in the Whitehead group $ \Wh(\pi):= \frac{K_1(\Lambda)}{\pm\pi} $ (see \cite{cohen} for a definition).\\

The $ s $-cobordism theorem, which can be found in \cite{kervaireBMS}, states that $ \tau (W,N_0)=0 $ is a sufficient condition\footnote{Trivially, the condition $ \tau (W,N_0)=0 $ is also necessary for $ W\dif N_0\x [0,1] $.} for $ W $ being diffeomorphic to $ N_0\x [0,1] $, provided $ n\geq 5 $.\\

We can reformulate this theorem into a statement about non triviality of some functional space: consider $ \cF $ the space of $ \cC^{\infty} $-functions $ f:W\to [0,1] $ such that $ f^{-1}(i)=N_i, i=0,1 $ with the $ \cC^{\infty}-$topology. Its subspace $ \cE $ consisting of functions without critical points is non-empty if and only if $ W\dif N_0\x [0,1] $, as it suffices to pick some $f\in\cE $ and to integrate the vector field $ \frac{\grad f}{\norm{\grad f}}$ relative to some Riemannian metric on $ W $ in order to find a diffeomorphism from $ N_0\x [0,1] $ to $ W $. We obtain so:\\

\begin{Teo}[\emph{Functional formulation of the $ s $-cobordism theorem}]\label{th:s-cob}
Let $ n\geq 5 $,
\[ \cE\neq\V \esp\Sii \esp\tau (W,N_0)=0  \]\qed
\end{Teo}

\begin{rem}\label{rem:u=0}
{\rm The relative homology $ H_*(W,N_0) $ vanishes since $ i_0 $ is a homotopy equivalence. This is indeed a necessary condition for $ \cE\neq\V $, since the Morse complex $ C_*(f) $ of a Morse function $ f\in\cE $ is zero in every degree, and the homology of $ C_*(f) $ is isomorphic to $ H_*(W,N_0) $ (see \cite{milnorMorse}).\\
By Lefschetz duality, we deduce that $ H^1(W,N_1)\iso H_n(W,N_0) $ vanishes. The same holds for $ H^1(W,N_0)\iso H_n(W,N_1) $ by using the fact that $ i_1 $ is also a homotopy equivalence.}
\end{rem}

We are going to consider the $ s $-cobordism theorem and the one from Latour as statements about the relative cohomologies $ H^1(X;Y) $ and $ H^1(X;Z) $ of a triad $ (X;Y,Z) $. Since the only relative cohomology class of degree $ 1 $ to consider in the case of an $ h $-cobordism is $ u=0 $, we will talk about \emph{the exact case} to refer to the context of this section.

\section{The theorem of Latour}\label{sec:Lat}

Consider now a closed manifold $ M^{n+1} $. We ask $ M $ to fiber over the circle $ \bS^1 $, which is equivalent by Tischler's theorem \cite{tischler} to the existence of a non-singular closed $ 1 $-form on $ M $.\\

We say that a cohomology class is non-singular if it is representable by a non-singular closed $1$-form. It is clear that there is no chance for $u=0\in H^1(M;\bR) $ to be non-singular since $M$ is closed. Latour's theorem characterizes degree one de Rham cohomology classes $0\neq u $ that are non-singular. Within the context of this section, here is the statement:

\begin{Teo}[\neee\cite{latour}]\label{th:Lat}
Let $ n\geq 5 $, and let $ \Omega^u_{NS} $ denote the space of non-singular closed $ 1 $-forms representing $ u $. We have:
\[ \Omega^u_{NS}\neq \V \esp\Sii\esp\left\lbrace \begin{array}{l}
H_*(M,-u)=0,\\
\tau(-u)=0,\\
u\text{ and } -u \text{ are stable.}
\end{array} \right. \,  \]\qed
\end{Teo}

Notice that a $\al\in\Omega^u_{NS}$ determines a whole ray $ru=[r\al],r\in\bR^*$ of non-singular cohomology classes. These form so a cone into $H^1(M;\bR)$. In particular $\Omega^u_{NS}\neq \V \Sii \Omega^{-u}_{NS}\neq \V$.\\

A degree one cohomology class can be seen as a morphism $ u:\pi\to\bR  $ just by integrating representatives of loops in $ M $. The Novikov ring associated to $ u $, denoted by $ \Lambda_u $, is a completion of the group ring $ \Lambda $. Elements of $ \Lam_u $ are formal sums $ \lam:=\sum n_ig_i,n_i\in\bZ $ such that, for every fixed $ C\in\bR $, there are only finitely many terms $ g_i $ verifying $ u(g_i)<C $. The homology $ H_*(M;-u) $ which appears in Latour's theorem is the Novikov homology, which was first constructed in \cite{novikovMulti}. The Novikov complex is the free finite $ \Lam_{-u} $-module $ \parent{\cN^{-u}_*:=\Lam_{-u}\tens_{\Lam}S_*(\wt{M}),\p_*} $, where $ S_*(\wt{M}) $ denotes the simplicial/cellular chain complex of the universal cover of $ M $ associated to a given triangulation/cell structure on $ M $.

\begin{rem}\label{rem:signe}
{\rm Is important to notice that Latour's theorem, which is a property of the cohomology class $u$, is stated in terms related to $\Lambda_{-u}$-modules.}
\end{rem}

The second right-side condition of theorem \ref{th:Lat} contains indeed the first: in order to define the torsion $ \tau(-u) $, we need the Novikov complex to be acyclic. In this case, $\tau(-u)$ is defined as follows: by setting a base of $ \cN^{-u}_* $, we obtain a contraction $ \de_*:\cN^{-u}_*\to\cN^{-u}_{*+1} $ as in \cite[\S4]{maumaryTypesimple}. The map $ \parent{\p+\de}_*:\cN^{-u}_{\ev}\to\cN^{-u}_{\odd} $ is then an isomorphism and we can consider $ S $, the class in $ K_1(\Lam_{-u}) $ of its associated matrix in the fixed basis. This class may depend on the choice of the basis (compare to \cite[\S7]{milnorWhtorsion}); in order to remove this indeterminacy, Latour defined the Whitehead group associated to $ -u $ as $ \Wh(-u):=\frac{K_1(\Lambda_{-u})}{T_{-u}}$, where the class $ [S] $ depends only on $ -u $. Here, $ T_{-u}:=\pm\pi\cdot\parent{1+\bparent{u<0}}\subset \Lambda_{-u}^{\x} $ is the subgroup of the so-called trivial units.\\
The torsion $\tau(-u)$ is defined by $[S]\in\Wh(-u) $.\\

An explanation about the stability condition of $\pm u$ is postponed to subsection \ref{ssec:stab}.\\

As he pointed out in his introduction, Latour's strategy to prove theorem \ref{th:Lat} is similar to that of the $ s $-cobordism theorem; the goal of the present paper is to show that Latour's theorem is indeed a natural generalization of $ s $-cobordism theorem for relative cohomology classes.

\section{A generalization framework}\label{sec:framework}
In Latour's theorem, the notion of $ u $-stability is related with \emph{unbounded} primitives of $ p^*(u) $ where $ p:\wh{M}\to M $ is the abelian cover of $ u $ having $ \pi_1(\wh{M}) $ equal to $ \ker(u) $. If we try to extend this notion to a null class $ u=0 $, the cover coincides with $ \Id: M\to M $ and we have no unbounded primitives of $ 0 $.\\
However, we only want to extend the notion of $ u $-stability for null classes of the relative $ 1 $-cohomology of an $ h $-cobordism.  We replace so the notion of $ h $-cobordism in the most trivially possible way in order to have unbounded primitives in the exact context when $ u=0\in H^1(W,N_0) \cup H^1(W,N_1)$:

\begin{df}\label{df:ext-triad}
From any $ h $-cobordism $ (W;N_0,N_1) $, we construct the triad $ (W_{\pm};N_-,N_+) $ by setting:
\begin{itemize}
\item $ N_-:=N_0\x(-\infty,0], N_+:=N_1\x [1,\infty) $ and
\item $ W_{\pm}:=N_-\coprod\limits_{\Id_{N_0}} W \coprod\limits_{\Id_{N_1}} N_+ $.
\end{itemize}
We call $ (W_{\pm};N_-,N_+) $ the \emph{extended triad} of $  (W;N_0,N_1) $.
\end{df}

In particular the cohomologies of an $ h $-cobordism and of its extended triad are the same and $ W $ is trivial if and only if $ W_{\pm} $ is diffeomorphic to $ N_0\x\bR $. We can so state the $ s $-cobordism theorem in terms of extended triads.

\begin{rem}\label{rem:simply}
{\rm Of course, the extended triad is not strictly an $ h $-cobordism since $ W_\pm $ has no boundary, but the inclusion $ i:(W;N_0,N_1)\plon (W_{\pm};N_-,N_+) $ is nevertheless a simple homotopy equivalence: any cell of, say $ N_- $, is of the form $ \Delta\x\bR^-  $ where $ \Delta $ is a cell of $ N_0 $ and we have a natural collapse $ c:N_-\to N_0 $.}
\end{rem}

\section{Comparison of the two theorems}\label{sec:Comparison}
Let us study how Latour's conditions relative to $ u\in H^1(M;\bR)\prive\sing{0} $ of closed manifolds $ M $ degenerate to the $s$-cobordism theorem condition for extended triads of $ h $-cobordisms $ (W;N_0,N_1)  $ as in section \ref{sec:framework}.\\

Firstly, regard the closed manifold $M$ as the triad $(M;\V_-,\V_+)$ and the cohomology class as living in $u\in H^1(M;\bR)=H^1(M,\V_-;\bR)$. Latour's conditions applied to $-u$ should be regarded as a statement about $-u\in H^1(W,\V_+;\bR)$ since in this case, the associated Novikov complex is constructed using $\Lam_u$-modules instead of $\Lam_{-u}$-modules.\\

Secondly, consider the $ h $-cobordism replaced by its extended triad $ (W_{\pm};N_-,N_+) $ as in definition \ref{df:ext-triad}. We distinguish the null-elements of the relative cohomologies by setting $ H^1(W_{\pm},N_-)=\sing{+0} $ and $ H^1(W_{\pm},N_+)=\sing{-0} $.\\

Now we study what happens to Latour's conditions when they are interpreted relatively to the extended triad $ (W_{\pm};N_-,N_+) $ for $ u=+0\in H^1(W_{\pm},N_-) $:

\begin{itemize}
\item The Novikov homology $ H_*\bparent{(W_{\pm},N_-),-0} $ is computed from the complex $ \cN^{-0}_*$. This complex is $\Lam_{-0}\tens_{\Lam}\cS_*(\wt{W_{\pm}},\wt{N_-}) $ by definition, but the ring $ \Lam_{-0} $ trivially coincides with the group ring $ \Lam $ and hence the Novikov complex $ \cN^{-0}_* $ is nothing but $ \cS_*(\wt{W_{\pm}},\wt{N_-}) $. So $ H_*\bparent{(W_{\pm},N_-),-0}=H_*(\wt{W_{\pm}},\wt{N_-})$ which is isomorphic to $H_*(\wt{W},\wt{N_0})=0 $ since both pairs are homotopy equivalent.\\
The first condition of Latour is so trivially true for $ h $-cobordisms as we have noticed on remark \ref{rem:u=0}. 
\item Since the set of trivial units $ T_{-0}=\pm\pi $, the group $ \Wh(-0) $ defined by Latour reduces to the usual Whitehead group $ \Wh(\pi) $. The torsion $ \tau(-0) $ is $ \tau(W_{\pm},N_-) $, since  $\cN_*^{-0}=\cS_*(\wt{W_{\pm}},\wt{N_-}) $. But the latter torsion coincides with the  Milnor torsion $ \tau(W,N_0)  $ since the pairs $ (W_{\pm},N_-) $  and $ (W,N_0) $ are simply homotopy equivalent by remark \ref{rem:simply}.\\
The condition $ \tau(-0)=0 $ of Latour is so the equivalent condition of theorem \ref{th:s-cob} for an $ h $-cobordism to be trivial.
\end{itemize}

\begin{rem}
{\rm The corresponding statements about $ u=-0\in H^1(W_{\pm},N_+) $ yield the vanishing of the relative homology $ H_*(\wt{W},\wt{N_1}) $ and associated torsion $ \tau(W,N_1) $, which is an equivalent formulation of the $ s $-cobordism theorem.}
\end{rem}

Note that the previous observations do not need the notion of extended triad and can be applied to the $ h $-cobordism $ (W;N_0,N_1) $ directly. We have established so far that the first two conditions of Latour's theorem reduce to theorem \ref{th:s-cob} when applied to an $ h $-cobordism or to its extended triad. We need so to prove that the third condition relative to stability holds trivially when reducing to $ u=\pm 0 $. This will be proved below in proposition \ref{pro:0-stability}, where the convenience of the concept of extended triad will become more apparent. 

\subsection{The stability condition}\label{ssec:stab}
To prove his theorem, Latour showed that every Morse closed $ 1 $-form $ \al $ representing $ u $ gives raise to a complex $ C_*(\al) $ of $ \Lam_{-u} $-modules which is simply equivalent to the Novikov complex $ \cN_*^{-u} $. The two first conditions that we have analized allow one to proceed as in the $ s $-cobordism theorem in order to recurrently eliminate zeros of index/coindex $ i $ by eventually adding zeroes of index/coindex $ i+2 $, apart from the case $ i=2 $ which is special. Adding $ \pm u $-stability, Latour obtained a sufficient condition to handle with this special case (compare with \cite{damian}). Since critical points of index/coindex $ 2 $ do not represent a natural obstruction in the exact case, $ \pm 0 $-stability should hold trivially. Let us recall what $ u $-stability means, as in \cite[\S 5.4]{latour}:\\

Consider $ p:\wh{M}\to M $ the covering whose fundamental group is $ \ker u $. Its transformation group is $ \frac{\pi}{\ker u}\iso\bZ^{\irr(u)} $. Since the class $ p^*(u) $ vanishes, any closed $ 1 $-form $ \al $ representing $ u $ admits a \emph{primitive}: a function $ \wh{f}:\wh{M}\to\bR $ verifying $ d\wh{f}=p^*(\al) $ and $ \wh{f}(g\cdot x)=u(g)+\wh{f}(x) $ for every pair $ (g,x)$ in $ \bZ^{\irr(u)}\x\wh{M}$. It is easy to see that for every $ t\in\bR, \wh{f}^{-1}\bparent{[t,\infty)} $ has only one connected component where $ \wh{f} $ is not bounded; denote it by $ \wh{M}_t $. The inclusions $ \bparent{\wh{M}_s\emb{i^s_t} \wh{M}_t}_{s>t} $ induce a projective system $ \cP(u):=\parent{\pi_1(\wh{M}_t)}_{t\in\bR} $.\\
Latour showed that this system does not depend on the choice of $\wh{f}$ but only on $ u $, up to projective isomorphism (see \cite[Lemme 5.7]{latour}). The $ u $-stability is a condition about $ \cP(u) $.

\begin{df}\label{df:u-stability}
A cohomology class $ u\in H^1(M;\bR) $ is \emph{stable} if there exists an increasing sequence $ (t_n)_{n\in\bN}\to\infty $ where the restrictions to the images of $ \cP(u) $ are isomorphisms. More precisely, if we set $ I_n:=\im\bparent{(\pi_1)_*(i^{t_{n+1}}_{t_n})} $ and  $j_n:=(\pi_1)_*(i^{t_{n+1}}_{t_n})|_{I_{n+1}}  $, then $ j_{n}:I_{n+1}\to I_n $ are isomorphisms for every $ n\in\bN $.
\end{df}

The next proposition shows how $ u $-stability reduces to a condition which holds trivially for extended triads of $ h $-cobordisms.

\begin{pro}\label{pro:0-stability}
The extended triad  $ (W_{\pm};N_-,N_+) $ of any $ h $-cobordism is $ \pm 0 $-stable.
\end{pro}
\begin{proof}
Let us deal with $(-0)$-stability. Here $ \sing{-0}=H^1(W_{\pm},N_+) $. In this situation $ \ker(-0) $ is identified with the whole $ \pi_1(W_{\pm},N_+) $, and the covering pair  $ (\wh{W_{\pm}},\wh{N_+}) $  to consider coincides with the pair  $ (W_{\pm},N_+) $ itself. By relative de Rham theory (see \cite[Ch.1,\S6]{bott+tu} for example), the class $ -0 $ is represented by the pair $ (df,f|_{N^+}) $ with $ f:W_{\pm}\to\bR $. We are free to choose $ f $ verifying $ f(x,t)=t$ for every $ (x,t)\in N_+\cup N_- $; since $ W\plon W_{\pm} $ is compact, there exists some $ 1\leq t_0\in\bR $ such that for every $ t\geq t_0 $, the unique unbounded component $ W_t $ of $ f^{-1}\bparent{[t,\infty)} $ equals $ N_1\x [t,\infty) $. The projective system $ \pi_1(W_t) $ is constantly $ \pi_1(N_1) $ with inclusions inducing the identity if $ t\geq t_0 $. By choosing any increasing sequence $ (t_n) $ starting at $ t_0 $, stability for the class $ -0  $ holds.
\end{proof}
\vspace{1cm}
\ack{This paper was conceived and written at Fall 2013, during a \emph{JSPS Short-term} post-doctoral fellowship. The author would like to express his gratitude to the Graduate School of Mathematics of Tokyo for their warm welcome and for having provided an excellent research environment.}
\newpage
\bibliography{Biblio}{}
\bibliographystyle{alpha}

{\sc Graduate School of Mathematical Sciences, the University of Tokyo\\
3-8-1 Komaba Meguro-ku, Tokyo 153-8914, Japan.\\
\textit{Telephone:} (+81) (0)3-5465-8292}\\
\textit{E-mail address:} \href{mailto:carlos@ms.u-tokyo.ac.jp}{carlos@ms.u-tokyo.ac.jp}

\end{document}